\documentclass[12pt]{article}

\usepackage{amsthm, amssymb, amstext}
\usepackage{comment}

\newtheorem{theorem}{Theorem}
\newtheorem{lemma}{Lemma}

\newtheorem{corollary}{Corollary}
\newtheorem{conjecture}{Conjecture}

\newtheorem{claim}{Claim}
\newtheorem{subclaim}{Subclaim}

\theoremstyle{definition}
\newtheorem{definition}{Definition}

\usepackage{amsthm}

\title{Reconfiguring colorings of graphs with \\ bounded maximum average degree}
\author{
Carl Feghali\thanks{Computer Science Institute of Charles University, Prague, Czech Republic, email: \texttt{feghali.carl@gmail.com} }
}

\begin{document}
\maketitle

\begin{abstract}
The reconfiguration graph $R_k(G)$ for the $k$-colorings
of a graph~$G$ has as vertex set the set of all possible $k$-colorings
of $G$ and two colorings are adjacent if they differ in the color of exactly
one vertex of $G$. Let $d, k \geq 1$ be integers such that $k \geq d+1$. We prove that for every $\epsilon > 0$ and every graph $G$ with $n$ vertices and maximum average degree $d - \epsilon$, $R_k(G)$ has diameter $O(n(\log n)^{d - 1})$. This significantly strengthens several existing results. 
\end{abstract}

\section{Introduction}

Let $k$ be a positive integer. A \emph{$k$-coloring} of a graph $G$ is a function $f: V(G) \rightarrow \{1, \dots, k\}$ such that $f(u) \not = f(v)$ whenever $(u, v) \in E(G)$.  The reconfiguration graph $R_k(G)$ for the $k$-colorings
of a graph~$G$ has as vertex set the set of all possible $k$-colorings
of $G$ and two colorings are adjacent if they differ in the color of exactly
one vertex of $G$.   

Given a non-negative integer $d$, a graph $G$ is \emph{$d$-degenerate} if every subgraph of $G$ contains a vertex of degree at most $d$. Expressed differently, $G$ is $d$-degenerate if there there exists an ordering $v_1, \dots, v_n$ of the vertices in $G$, called a \emph{$d$-degenerate ordering}, such that each $v_i$ has at most $d$ neighbors $v_j$ with $j < i$. The maximum average degree of a graph $G$ is defined as
\[
\max \bigg\{ \frac{2|E(H)|}{|V(H)|} : H \subseteq G \bigg\}. 
\]
In particular, if $G$ has maximum average degree strictly less than some positive integer $d$, then $G$ is $(d-1)$-degenerate.

Consider the following conjecture of  Cereceda \cite{luisthesis}. 

\begin{conjecture}\label{conj}
For every integers $k$ and $\ell$, $\ell \geq k + 2$, and every $k$-degenerate graph $G$ on $n$ vertices, $R_{\ell}(G)$ has diameter $O(n^2)$. 
\end{conjecture} 

The conjecture appears difficult to prove or disprove, with the case $k = 1$ only being known despite some efforts; for a recent exposition on the conjecture and the results surrounding it see \cite{eiben, heinrich}. The most important breakthrough is Theorem 1 in \cite{heinrich} due to Bousquet and Heinrich, which addresses a number of cases for Conjecture \ref{conj}, generalising several existing results. For instance, it is shown in \cite{heinrich} that there exists a constant $c > 0$ independent of $k$ such that $R_{\ell}(G)$ has diameter at most $(cn)^{k+1}$ for every $\ell \geq k + 2$.  

The purpose of this note is to prove the following theorem.

\begin{theorem}\label{main}
Let $d, k \geq 1$ be integers such that $k \geq d+1$. For every $\epsilon > 0$ and every graph $G$ with $n$ vertices and maximum average degree $d - \epsilon$, $R_k(G)$ has diameter $O(n(\log n)^{d - 1})$.
\end{theorem}

Theorem \ref{main} is a generalisation of \cite[Theorem 2]{bousquet11}. In particular, it has the following immediate consequences. By Euler's formula, planar graphs, triangle-free planar graphs and planar graphs of girth 5 have maximum average degrees strictly less than, respectively, $6$, $4$ and $7/2$. 
Hence Theorem \ref{main} affirms (and is stronger than) Conjecture \ref{conj} for planar graphs of girth 5 but is one color short of confirming the conjecture for planar graphs and triangle-free planar graphs. 
It nevertheless generalises some best known existing results. More precisely, our theorem subsumes both \cite[Corollary 5]{bousquet11} and  \cite[Theorem 1]{heinrich} restricted to planar graphs, as well as \cite[Corollary 7]{bousquet11} and \cite[Corollary 1]{feghaliplanar}. 

\section{The proof}

In this section, we prove Theorem \ref{main}. Our approach is essentially a combination of the ones found in \cite{heinrich, feghali}. We begin with some definitions.

\begin{definition} Given a graph $G$, a coloring $\alpha$ of $G$ and a subgraph $H$ of $G$, let $\alpha^H$ denote the restriction of $\alpha$ to $H$.   \end{definition} 

\begin{definition}Let $G$ be a graph, and let $k$ be a nonnegative integer. A subset $S \subseteq V(G)$ is a \emph{$k$-independent set of $G$} if $S$ is an independent set of $G$ and every vertex of $S$ has degree at most $k$ in $G$.
\end{definition}
 
\begin{definition}\label{def} For integers $s\geq 0$ and $t \geq 1$, a graph $G$ is said to have \emph{degree depth $(s, t)$} if there exists a partition $\{V_1, \dots, V_t\}$ of $V(G)$, called an \emph{$s$-degree partition}, such that $V_1$ is an $s$-independent set of $G$ and, for $i \in \{2, \dots, t\}$, $V_i$ is an $s$-independent set of $G \setminus \bigcup_{j = 1}^{i - 1} V_j$.  
\end{definition}

In what follows, let $G$ be a graph of degree depth $(s, t)$ and with $s$-degree partition $\{V_1, \dots V_t\}$.

\begin{definition}\label{def4} An ordering $v_n, \dots, v_1$ of $V(G)$ is said to be \emph{embedded in} $\{V_1, \dots, V_t\}$ if,  for every pair $(v_i, v_j) \in V(G) \times V(G)$ such that $v_i \in V_p$ and $v_j \in V_q$, $i < j$ implies $p \leq q$. 
\end{definition}

Notice that the ordering in Definition \ref{def4}  is an $s$-degenerate ordering of $G$. 

  If $H$ is a subgraph of $G$ such that $V(H) = \bigcup_{j = 1}^h V_j$ for some index $h \in \{1, \dots, t\}$, then $H$ is called a \emph{layered} subgraph of $G$, and $h$ is its \emph{boundary}. 
   
 In the next definition, we shall slightly abuse Definition \ref{def}.
 
 \begin{definition} If $H$ is a layered subgraph of $G$ with boundary $h$, then we say that $H$ has degree depth $(s', t)$ if, for each index $j \in \{1, \dots, h\}$, each $v \in V(H) \cap V_j$  has at most $s'$ neighbors in $\bigcup_{i = j + 1}^t V_i$.  
 \end{definition}

 We have the following crucial lemma. 
 
 \begin{lemma}\label{2deg}
 Let $s \geq 0$ and $ t \geq 1$ be  integers, let $G$ be a graph with degree depth $(s, t)$, and let $F$ be a layered subgraph of $G$. Any $(s + 2)$-coloring of $G$ can be recolored, using only colors $1, \dots, s + 2$, to some coloring of $G$ in which color $s+2$ is not used in $F$ by $O((s + 1)2^{s - 1}t^{s})$ recolorings per vertex of $F$ and by not recoloring any vertex of $G \setminus F$.  
 \end{lemma}
 
 \begin{proof}
Let $\{V_1, \dots, V_t\}$ be an $s$-degree partition of $G$, and let $V(F) = V_1 \cup \dots \cup V_b$, where $b \geq 1$ is the boundary of $F$. Let $v_m, \dots, v_1$ be an ordering of $V(F)$ that is embedded in $\{V_1, \dots, V_b\}$.  Let $\alpha$ be an $(s + 2)$-coloring of $G$, and let $h \in \{1, \dots, b\}$ be the \emph{smallest} index such that $V_h$ contains a vertex with color $s + 2$ under $\alpha$. 
 Let $W$ denote the subset of vertices of $V_h$ with color $s + 2$. For each color $a \in \{1, \dots, s + 1\}$, define $W_a$ to be the subset of $W$ whose vertices have no neighbor earlier in the ordering with color $a$. More formally, $$W_a = \{v_i \in W: \alpha(v_j) \not= a \text{ for all neighbors } v_j \text { of } v_i \text{ with } j > i\},$$ 
 and notice that $$W = \bigcup_{i = 1}^{s + 1} W_i.$$

 \begin{claim}\label{claim3}
 Let $U = \bigcup_{i=1}^{h - 1}V_i$. For each $a \in \{1, \dots, s + 1\}$, there is a sequence of recolorings in $R_{s + 2}(G)$ such that 
 \begin{itemize}
 \item each vertex of $U$ is recolored $O((2t)^{s - 1})$ times,
 \item each vertex of $W_a$ is recolored at most once,
 \item no vertex of $V(G) \setminus (U \cup W_a)$ is recolored, and
 \item at the end of the sequence, no vertex of $U \cup W_a$ has color $s + 2$. 
 \end{itemize}
 \end{claim}
 
Let us first show how to use the claim to prove the lemma. Applying the sequence described in Claim \ref{claim3} for each $a \in \{1, \dots, s + 1\}$, we obtain a coloring in which color $s + 2$ is not used in $U \cup V_h$ by $O((s + 1)(2t)^{s - 1})$ recolorings.  The smallest index $h'$ such that  $V_{h'}$ contains a vertex with color $s + 2$ has now increased; hence at most $b \leq t$ such repetitions are needed to obtain a coloring in which color $s + 2$ is not used in $F$,  so each vertex is recolored $O((s + 1)2^{s - 1}t^{s})$ times and the lemma follows.  It remains to prove the claim.

\begin{proof}[Proof of Claim \ref{claim3}] Let $F^* = F[U \cup W_a]$ and note that $F^*$ has degree depth $(s', t)$ for some $s' \in \{0, \dots, s\}$. We are going to apply induction on $s'$. 
  The base case $s' = 0$ is trivial (simply immediately recolor the vertices of $W_a$) so we can assume that $s \geq s' > 0$ and that Claim \ref{claim3} and hence, by the observation following the statement of Claim \ref{claim3}, also the lemma holds for each subgraph $K$ of $G$ and layered subgraph of $K$ of degree depth $(s' - 1, t)$. 
  
 In the inductive step, we are in fact going to establish the claim for the pair $((s', t), s + 2)$, where the first term of the pair corresponds to the degree depth of $F^*$ and the second term to the number of colors, assuming its validity for the pair $((s' - 1, t), s + 1)$. Let $u_k, \dots, u_1$ be an ordering of the vertices of $U$ that is embedded in $\{V_1, \dots, V_{h - 1}\}$.  We first try to recolor immediately, whenever possible, each vertex of $U$ to color $s + 2$ starting with $u_k$ and moving forward towards $u_1$. 
Let $\gamma$ denote the resulting coloring, let $S = \{\gamma(v) = s + 2: v \in V(G)\}$ and let $H = G[U \setminus S]$. 

\begin{subclaim}
$H$ has degree depth $(s'-1, t)$.
\end{subclaim}

\begin{proof}[Proof of Subclaim.]By our choice of $h$, each vertex $u \in U \cap V_p$ for some $p \in \{1, \dots, h - 1\}$ either satisfies $\gamma(u) = s + 2$ or has a neighbor $u' \in V_q$ for some $q \in \{p + 1, \dots, t\}$ such that $\gamma(u') = s + 2$. This implies the subclaim.
\end{proof} 

By the above subclaim, we can apply the induction hypothesis to the pair $((s' -1, t), s + 1)$ with $H$ and $G \setminus S$ playing the roles of $F$ and $G$, respectively.  This gives a sequence of recolorings (that uses only colors $1, \dots, s + 1$) from  $\gamma^H$ to some coloring $\zeta^H$ of $H$ such that
\begin{itemize}
 \item color $a$ is not used in $\zeta^H$,
 \item  the number of recolorings per vertex of $H$ is $O(2^{s - 2} t^{s - 1})$, and
 \item no vertex of $G \setminus (S \cup H)$ is recolored. 
 \end{itemize}
Clearly, this sequence of recolorings vacuously translates to a sequence of recolorings in $G$ from $\alpha$ to coloring $\zeta$ satisfying $\zeta(v) = \zeta^H(v)$ if $v \in V(H)$ and $\zeta(v) = \gamma(v)$ if $v \in V(G) \setminus V(H)$. From $\zeta$, we can now immediately recolor each vertex of $W_a$ to color~$a$. It remains to recolor each vertex of $U$ to a color distinct from $s + 2$. To do so, we simply repeat the above steps with the roles of $a$ and $s + 2$ interchanged. This takes again $O(2^{s - 2} t^{s - 1})$ recolorings per vertex of $H$. Hence each vertex of $H$ is recolored in total $O((2t)^{s-1})$ times. This proves the claim and hence completes the proof of the lemma. \end{proof}

   \end{proof}

We can prove our final lemma, from which Theorem \ref{main} follows easily.

\begin{lemma}\label{l}
 Let $s \geq 0$ and $t \geq 1$ be integers, and let $G$ be a graph with $n$ vertices and degree depth $(s, t)$. Then $R_{s + 2}(G)$ has diameter $O(ns(2t)^{s})$.
\end{lemma}

\begin{proof}
As before, we proceed by induction on the pair $((s, t), s + 2)$, where the first term corresponds to the degree depth of $G$ and the second term to the number of colors. The base case $s = 0$ is trivial, so we can assume that $s > 0$ and that the lemma holds for the pair $((s-1, t), s + 1)$.  

Let $\alpha$ and $\beta$ be two $(s + 2)$-colorings of $G$, and let $\{V_1, \dots, V_t\}$ be an $s$-degree partition of $G$. It suffices to show that we can recolor $\alpha$ to~$\beta$ by $O(s(2t)^{s})$ recolorings per vertex. By Lemma \ref{2deg} with $F = G$, we can recolor $\alpha$ to some $(s + 1)$-coloring $\alpha_1$ of $G$ and $\beta$ to some $(s + 1)$-coloring $\beta_1$ of $G$ by $O(s2^st^{s})$ recolorings per vertex. 

Let $v_n, \dots, v_1$ be an ordering of $V(G)$ that is embedded in $\{V_1, \dots, V_t\}$. We recolor $\alpha_1$ and $\beta_1$ to new colorings $\alpha_2$ and $\beta_2$ of $G$ by trying to recolor, from $\alpha_1$ and $\beta_1$,  immediately whenever possible  each vertex of $G$ to color $s + 2$ starting with $v_n$ and moving forward towards $v_1$.
Let $S = \{v \in V(G): \alpha_2(v) = s + 2 (=\beta_2(v)) \}$. As before, the graph $H = G - S$ has degree depth $(s - 1, t)$. So we can apply our induction hypothesis to recolor $\alpha_2^H$ to $\beta_2^H$ by $O((s-1)(2t)^{s-1})$ recolorings per vertex using only colors $1, \dots, s+1$  (as this sequence of recolorings does not use color $s + 2$, we need not worry about adjacencies between $H$ and $S$). This completes the proof. 
\end{proof}

 \begin{proof}[Proof of Theorem \ref{main}]
 Let $H$ be any subgraph of $G$, and let $h = |V(H)|$. An independent set $I$ of $H$ is said to be \emph{special} if $I$ is  a $(d-1)$-independent set of $H$ and $|I| \geq \epsilon h / d^2$.
 It was shown in \cite{feghali} that $H$ contains a special independent set. This means that there is a partition $\{I_1, I_2, \dots, I_{\ell}\}$ of $V(G)$ such that $I_1$ is a special independent set of $G$ and, for $i \in \{2, \dots, \ell\}$, $I_i$ is a special independent set of $G \setminus \Big(\bigcup_{j = 1}^{i - 1} I_j\Big)$.  Thus $G$ has degree depth $(d - 1, \ell)$.  As $\ell = f(n)$ satisfies the recurrence $$f(n) \leq f\Big(n - \frac{\epsilon n}{d^2}\Big) + 1,$$  it follows that $\ell = O(\log n)$, by the master theorem. The theorem now follows by Lemma \ref{l} with $t = \log n$ and $s = d - 1$. 
 \end{proof} 
 
 Similarly, we can slightly improve on the constant $c$ in the aforementioned main result from \cite{heinrich}.
 
 \begin{corollary}
Let $k, n \geq 1$ be integers, and let $G$ be a $k$-degenerate graph with $n$ vertices. Then $R_{k + 2}(G)$ has diameter $O(2^kn^{k + 1})$. 
 \end{corollary}
 
 \begin{proof}
Noting that every $k$-degenerate graph with $n$ vertices has degree depth $(k, n)$, the corollary immediately follows from Lemma \ref{l}. 
 \end{proof}

 \section*{Acknowledgements}
The author is indebted to the referee for spotting several inaccuracies and for many suggestions that significantly improved the presentation of the paper. 
This work was partially supported by the Research Council of Norway via the project CLASSIS  grant
number 249994 and by grant 19-21082S of the Czech Science Foundation.

\bibliography{bibliography}{}
\bibliographystyle{abbrv}
 
\end{document}